\documentclass[11pt]{amsart}

\usepackage{amsmath,amssymb,graphics,amsthm,amscd}
\usepackage{longtable}
\usepackage{cite,xypic}
\usepackage{color}
\usepackage{graphicx}
\usepackage{url}
\usepackage{epsf}
\usepackage{fancyhdr}
\usepackage{lscape}
  \renewenvironment{thebibliography}[1]{
    \begin{oldthebibliography}{#1}
      \setlength{\parskip}{0ex}
      \setlength{\itemsep}{0ex}
  }
  {
    \end{oldthebibliography}
  }
  
\setlongtables

\DeclareMathOperator{\Hom}{Hom}
\DeclareMathOperator{\Ext}{Ext}

\newcommand{\Q}{\mathbb{Q}}
\newcommand{\NN}{\mathbb{N}}
\newcommand{\R}{\mathbb{R}}

\newcommand{\Z}{\mathbb{Z}}

\newcommand{\up}{\uparrow}
\newcommand{\rarr}{\rightarrow}

\DeclareMathOperator{\ch}{ch}

\DeclareMathOperator{\rad}{rad}
\DeclareMathOperator{\p}{\mathfrak P}
\DeclareMathOperator{\Stab}{Stab}
\DeclareMathOperator{\Ind}{Ind}

\newcommand{\half}{\frac{1}{2}}

\newcommand{\alphac}{\alpha^\vee}

\newcommand{\betac}{\beta\check{\ }\,}

\begin{document}

\newcounter{rownum}
\setcounter{rownum}{0}

\newtheorem{lemma}{Lemma}[section]
\newtheorem{theorem}[lemma]{Theorem}
\newtheorem*{TA}{Theorem A}
\newtheorem*{TB}{Theorem B}
\newtheorem*{TC}{Theorem C}
\newtheorem*{C3}{Corollary 3}
\newtheorem*{T4}{Theorem 4}
\newtheorem*{C5}{Corollary 5}
\newtheorem*{claim}{Claim}
\newtheorem{corollary}[lemma]{Corollary}
\newtheorem{conjecture}[lemma]{Conjecture}
\newtheorem{prop}[lemma]{Proposition}
\theoremstyle{remark}
\newtheorem{remark}[lemma]{Remark}
\newtheorem{obs}[lemma]{Observation}
\theoremstyle{definition}
\newtheorem{defn}[lemma]{Definition}

  \def\hal{\unskip\nobreak\hfil\penalty50\hskip10pt\hbox{}\nobreak
  \hfill\vrule height 5pt width 6pt depth 1pt\par\vskip 2mm}

\newenvironment{changemargin}[1]{%
  \begin{list}{}{%
    \setlength{\topsep}{0pt}%
    \setlength{\topmargin}{#1}%
    \setlength{\listparindent}{\parindent}%
    \setlength{\itemindent}{\parindent}%
    \setlength{\parsep}{\parskip}%
  }%
  \item[]}{\end{list}}

\parindent=0pt
\addtolength{\parskip}{0.5\baselineskip}

 \title[First cohomology groups]{First cohomology groups for finite groups of Lie type in defining characteristic}
\author{Alison E. Parker}
\address{School of Mathematics\\
University of Leeds\\
Leeds, LS2 9JT, UK\\
}
\email{a.e.parker@leeds.ac.uk {\text{\rm(Parker)}}}
\author{David I. Stewart}
\address{New College, Oxford\\ OX1 3BN, UK}
\email{david.stewart@new.ox.ac.uk {\text{\rm(Stewart)}}}

\pagestyle{plain}
\begin{abstract}Let $G$ be a finite group of Lie type, defined over a
  field $k$ of characteristic $p>0$. We find explicit bounds for the
  dimension of the first cohomology group for $G$ with coefficients in
  a simple $kG$-module. We proceed by bounding the number of
  composition factors of Weyl modules for simple algebraic groups
  independently of $p$ and using this to deduce bounds for the
  $1$-cohomology of simple algebraic groups. Finally, we use this to
  obtain estimations for the growth rate of the maximum dimension
  $\{\gamma_l\}$ of these $1$-cohomology groups over all groups of Lie
  type of rank $l$. We find that $\log\gamma_l$ is $O(l^3\log l)$ (or
  if the Lusztig conjecture holds, $O(l^2\log l)$).
\end{abstract}
\maketitle
\section{Introduction}
In \cite{Gur86} R. Guralnick made a conjecture that there should
be a universal bound on the dimension of the first cohomology groups
$H^1(G,V)$ where $G$ is a finite group and $V$ is an absolutely
irreducible faithful representation for $G$. The conjecture reduces to the case
where $G$ is a finite simple group, thence readily to the case where
$G$ is a finite group of Lie type or alternating group, since the
possibilities for $G$ and $V$ when $G$ is sporadic and $H^1(G,V)\neq
0$ is finite.

Very recently, computer calculations of Frank L\"ubeck have given evidence that the Guralnick conjecture may unfortunately be false. It even seems possible that the growth rate of the sequence with $l$
is in fact exponential. Finding a sequence of increasing dimensions of $H^1(G,V)$ has proven to be very challenging, but assuming such a sequence is eventually produced, the  natural question, as originally proposed in \cite{Sco03}
is to find upper and lower bounds for the growth rate of $\dim H^1(G,V)$ as
the Lie rank of $G$ grows.

For the question of finding upper bounds, there are two distinct cases, admitting very different techniques. In defining characteristic---that is where $V$ is a representation for a finite
group of Lie type $G(p^r)$ over a field of characteristic $p$---the
result \cite[Theorem 7.10]{CPS09} uses algebraic group methods to
assert the implicit existence of a bound on 1-cohomology.\footnote{In
  \cite{PS11} this was generalised in several directions, notably to
  show that the same bound works for the dimension of $\Ext^1$ between
  simples.}  In cross-characteristic, \cite{GT11} is more specific,
giving $|W|+e$ as a bound on $\dim H^1(G,V)$, where $e$ is the twisted
rank of $G$ and $W$ is the Weyl group $W(\Phi)$ of the root system
$\Phi$ of $G$. If the alternating groups are thought of as Lie type groups over the field of one element, then they are also covered by this case.

The Cline--Parshall--Scott result for defining characteristic 
suffers in comparison to the cross-characteristic case by not
furnishing any explicit bounds. The reason is that it relies on a
large piece of machinery due to Andersen--Jantzen--Soergel: if $G$ is
a semisimple algebraic group defined over a field of characteristic $p$, then
there is a prime $p_0$ such that for all $p\geq p_0$ a significant
amount of the representation theory of $G$ is independent of $p$,
including a character formula for all restricted simple modules. Using
this theorem one can reduce to the case of dealing with each $p$
separately: \cite{CPS09} shows that there are only cofinitely many values of $p$ for
which the maximum value of $\dim H^1(G,V)$ may differ. Unfortunately, in
the original result, $p_0$ was implicit. Nowadays one does know bounds
for $p_0$, courtesy of Peter Fiebig \cite{Fie}, but these are simply
too big: in combination with the \cite{CPS09} result, it leads only to
a bound for $\dim H^1(G,V)$ which grows super-exponentially with the rank.

The main purpose of this paper is to find a new proof in the defining
characteristic case. Our proof has certain advantages over the
previous one. First of all, we get, as in \cite{GT11}, explicit bounds; secondly, our proof is quite direct, and
uniform over all $p$. In particular we make no use of the Lusztig
character formula, nor of the representation theory of quantum groups.

Certain aspects of our proof are similar to that in \cite{CPS09}. We
start by finding a uniform bound for $H^1(G,V)$ with $G$ a simple
algebraic group of fixed root system and $V$ a simple $G$-module; for this our innovation is to
make use of the sum formula (Theorem \ref{sumformT})---one of the few
tools available in the the theory which is uniform with $p$. We
 use the sum formula to bound the length (i.e.~number of composition
 factors) of a Weyl module $V=V(\lambda)$ with restricted high weight
 $\lambda\in X_1$. If $L(\lambda)$ is the corresponding simple head of
 $V$, it is a well known fact that $H^1(G,L(\lambda))\cong \Hom_G(\rad
 V,k)$, showing that a bound on the number of composition factors of
 $V$ also bounds the dimension of $H^1(G,L(\lambda))$. After that, a Frobenius kernel
 argument (Proposition \ref{theprop}) gives us a bound for all simples (not just the restricted
 ones). From here we follow the same route as \cite{CPS09} by using
 results of Bendel--Nakano--Pillen to relate algebraic group
 cohomology to finite group cohomology.
 
Our main result is
\begin{TA}Let $G$ be a finite simple group of Lie type with associated
  Coxeter number $h$ and defined over an algebraically closed field $k$. Let $V$ be an irreducible $kG$-module.  Then \[\dim H^1(G,V)\leq
  \max\left\{\frac{z_p^{h^3/6+1}-1}{z_p-1},\frac{1}{2}\left(h^2(3h-3)^3\right)^{\frac{h^2}{2}}\right\} \]
  where $z_p=\lfloor h^3/6 (1+\log_p(h-1))\rfloor \leq\lfloor h^3/6
  (1+\log_2(h-1))\rfloor$; if $p\geq h$ then we may take
  $z_p=h^3/6$.\end{TA}
 
Applying this result together with that of Guralnick--Tiep, we may state a growth rate result. Let $\{\gamma_l\}$ be
the sequence $\{\max H^1(G,V)\}$, where the maximum is over all finite simple
groups of Lie type of Lie rank $l$ and irreducible representations $V$ for $G$. 

\begin{TB}We have $\log\gamma_l=O(l^3\log l)$.\end{TB}

In the presence of the Lusztig conjecture we can do better than
this. At the end of the paper we can use our approach from bounding the length of Weyl modules  to improve the
Cline--Parshall--Scott bounds for specific $p$, removing a
dependence on the Kostant partition function. We get
\begin{TC} Whenever $p\geq h$, suppose the Lusztig Character Formula
  holds for all $p$-restricted weights. Let $G$ be a finite group of Lie type defined over an algebraically closed field $k$. Then if $V$ is an irreducible $kG$-module, we have \[\dim H^1(G,V)\leq
  \max\left\{(2h)^{h^2/2},\frac{1}{2}\left(h^2(3h-3)^3\right)^{\frac{h^2}{2}}\right\}.\]
  In particular, $\log\gamma_l=O(l^2\log l)$.\end{TC}

\section{Notation}
Most of our notation will be that used in Jantzen \cite{Jan03} and the
reader is referred there for the proper definitions.

Throughout this paper $k$ will be an algebraically closed field of
characteristic $p$ and $G$ will be a group, finite or algebraic. 

Suppose $G$ be a connected, simple, simply connected algebraic group. 
We fix  a maximal torus $T$ of $G$ of dimension $n$, 
the rank of $G$.
We also fix $B$, a Borel subgroup of $G$ with $B \supseteq T$ 
and let $W$ be the Weyl group of $G$.

Let $X(T)=X$ be the weight lattice for $G$ and $Y(T)=Y$ the dual weights.
The natural pairing $\langle -,- \rangle :X  \times Y
\rarr \Z$ is bilinear and induces an isomorphism $Y\cong \Hom_\Z( X,\Z)$.
We take $\Phi$ to be the roots of $G$. 
For each $\alpha \in \Phi$ we take 
$\alphac \in Y$ to be  the coroot of $\alpha$. 
Let  $\Phi^+$ be the positive roots, chosen so that $B$ is the negative
Borel and let $\Pi$ be the set of simple roots. 
Set $\rho = \half \sum_{\alpha \in \Phi^+} \alpha \in X(T)\otimes_{\Z}\Q$.

We have a partial order on $X(T)$ defined by 
$\mu \le \lambda \iff \lambda -\mu \in \NN S$.
A weight $\lambda$ is \emph{dominant} if
$\langle \lambda, \alphac \rangle \ge 0$ for all $\alpha \in \Pi$ and
we let $X^+(T)=X^+$ be the set of dominant weights.

Take $\lambda \in X^+$ and let $k_\lambda$ be the one-dimensional module
for $B$ which has weight $\lambda$. We define the induced
module, $H^0(\lambda)= \Ind_B^G(k_\lambda)$. This module has
simple socle $L(\lambda)$, 
the irreducible $G$-module of highest weight
$\lambda$. One also has the Weyl module $V(\lambda)=H^0(-w_0(\lambda))^*$ with $L(\lambda)$ as its simple head. Both modules have formal character given by Weyl's character formula,
$\chi(\lambda)$.  Any finite dimensional, rational irreducible $G$-module
is isomorphic to $L(\lambda)$ for a unique $\lambda \in X^+$.


We return to considering the weight lattice $X(T)$ for $G$.
There are also the affine reflections
$s_{\alpha,mp}$ for $\alpha$ a positive
root and $m\in \Z$ which act on $X(T)$ as
$s_{\alpha,mp}(\lambda)=\lambda -(\langle\lambda,\alphac\rangle -mp )\alpha$.
These generate the affine Weyl group $W_p$. 
We mostly use the dot action of $W_p$ on $X(T)$ which is 
the usual action of $W_p$, with the origin
shifted to $-\rho$. So we have $w \cdot \lambda = w(\lambda+\rho)-\rho$.
If $C$ is an alcove for $W_p$ then its closure $\bar{C}\cap X(T)$ is a
fundamental domain for $W_p$ operating on $X(T)$. The group $W_p$
permutes the alcoves simply transitively.
We set 
$C_\mathbf Z= \{ \lambda \in X(T) \otimes _{\Z} \R \ \mid\  0< \langle \lambda +\rho,
\alphac \rangle < p \quad \forall\, \alpha \in \Phi^+ \}$ 
 and call $C_\mathbf Z$ the \emph{fundamental alcove}.
We also set
$h = \max \{ \langle \rho, \betac \rangle +1 \ \mid\ \beta \in
\Phi^+\}$, the Coxeter number of $\Phi$. We have
$C_\mathbf Z \cap X(T) \ne \emptyset \ \iff \ \langle \rho, \betac \rangle < p,
\  \forall\, \beta \in \Phi^+ \ \iff \ p \ge h$.

We say that $\lambda$ and $\mu$ are \emph{linked} if  they belong to the 
same $W_p$ orbit on $X(T)$ (under the dot action).
If two irreducible modules $L(\lambda)$ and $L(\mu)$ are in the same $G$
block then $\lambda$ and $\mu$ are linked. 
Linkage gives us 
another partial order on both $X(T)$, and the set of alcoves of $X(T)$, denoted $\up$.
If $\alpha$ is a positive root  and $m\in \Z$ then we set
$$s_{\alpha,mp} \cdot \lambda \up \lambda \quad\mbox{if and only
if}\quad \langle\lambda+\rho ,
\alphac\rangle \ge mp.$$
This then generates an order relation on $X(T)$.
So $\mu \up \lambda$ if there are reflections $s_i \in W_p$ with 
$$\mu=s_m s_{m-1} \cdots s_1 \cdot \lambda \up 
s_{m-1} \cdots s_1 \cdot \lambda \up 
\cdots \up
s_1 \cdot \lambda \up 
\lambda.$$ The notion $C_1\up C_2$ is defined similarly.

For $\lambda\in X(T)$, define $n_\alpha,d_\alpha\in \mathbb Z$ by $\langle\lambda+\rho,\alpha^\vee\rangle=n_\alpha p+d_\alpha$ and $0<d_\alpha\leq p$. Following \cite[II.6.6, Remark]{Jan03}, we define  \begin{equation}\label{sumdlambda}d(\lambda) = \sum_{\alpha\in
    \Phi^+} n_\alpha=\sum_{\alpha\in
  \Phi^+}\left\lfloor\frac{\langle\lambda+\rho,\alpha^\vee\rangle}{p}\right\rfloor,\end{equation} where $\lfloor r\rfloor$ denotes the greatest integer function for
$r\in \mathbb R$.  Recall from \cite[II.6.6, Remark]{Jan03}, that each
  $\lambda\in X(T)$ lies in the upper closure $\widehat C$ of a unique
  alcove $C$. If $\lambda\in X^+$ and $d(C)$ is defined as the longest chain
  $C_\mathbf Z\uparrow C_1\uparrow\dots C_n=C$, then we have
  $d(\lambda)=d(C)$.

When $G$ is algebraic or finite, the category of (rational) $kG$-modules has enough injectives and so we may
define $\Ext_G^*(-,-)$ as usual by using injective
resolutions. We define $H^i(G,V) := \Ext_G^i(k,V)$.

\section{Preliminaries}

The following famous result is \cite[II.8.19]{Jan03} and is due to Andersen in its full generality.

\begin{theorem}[The sum formula]\label{sumformT} For each $\lambda\in
  X^+$ there is a
  filtration of $G$-modules
\[V(\lambda)=V(\lambda)^0\supset V(\lambda)^1\supset V(\lambda)^2\supset\dots\]
such that
\begin{equation}\label{sum}\sum_{i>0}\ch V(\lambda)^i=\sum_{\alpha\in
    \Phi^+}\sum_{0<mp<\langle\lambda+\rho,\alpha^\vee\rangle}v_p(mp)
\chi(s_{\alpha,mp}\cdot\lambda)\end{equation}
and \begin{equation}V(\lambda)/V(\lambda)^1\cong
  L(\lambda)\end{equation}
where $v_p$ is the usual $p$ valuation.\end{theorem}

Here the definition of formal character $\chi(\lambda)$ for dominant
$\lambda \in X^+$ is extended to all weights 
$\nu\in X(T)$ via the formula  
  $\chi(w\cdot\nu)=\det(w)\chi(\nu)$ for $w\in W$, 
  \cite[II.5.9(1)]{Jan03}. 
If $w\cdot\nu$ is not dominant for any
  $w\in W$, then $\chi(\nu)=0$. Otherwise we can express $\chi(\nu)$
  as $\pm\chi(\mu)$ for $\mu\in X^+$. 

\begin{lemma}\label{estimates} Suppose $\lambda\in X^+$, with
  $\langle\lambda+\rho,\alpha_0^\vee\rangle\leq b$.

(i) The number of terms arising from expansion of the two summations
  of the right-hand side of (\ref{sum}) is $d(\lambda)$.

(ii) The maximum value of $v_p(mp)$ occurring in the summands of
  (\ref{sum}) is \[\left\lfloor\log_p (b-1)\right\rfloor.\]
  
  (iii) After rewriting  each character
  $\chi(s_{\alpha,mp}\cdot\lambda)$ of the RHS of 
  (\ref{sum}) in terms of dominant weights $\chi(\mu)$  
  for some $\mu\in X^+$ and collecting like terms, any remaining
  $\chi(\mu)$ with
  a non-zero coefficient has
  $d(\mu)<d(\lambda)$.
  \end{lemma}

\begin{proof}(i) This is immediate from equation (\ref{sumdlambda}).

(ii) As $\lambda$
  is dominant, one has $\langle\lambda+\rho,\alpha^\vee\rangle\leq
  \langle\lambda+\rho,\alpha_0^\vee\rangle$ for any
  $\alpha\in\Phi^+$. The result is now clear.
  
  (iii) 
As noted above  we can express the non-zero $\chi(\nu)$ in the RHS of
\eqref{sum}
  as $\pm\chi(\mu)$ for $\mu\in X^+$. 
  
  Now the strong linkage principle implies that any composition factor
  $L(\mu)$ of $V(\lambda)$ satisfies $\mu\uparrow\lambda$. It follows
  that after complete expansion and rewriting of the terms in
  (\ref{sum}) as $\chi(\mu)$ with $\mu\in X^+$ (as above), $\mu$ must
  satisfy $\mu\uparrow\lambda$. Now it is immediate that
  $d(\mu)<d(\lambda)$.
    \end{proof}
  
In \cite{Boe01} the maximum value of $d(\lambda)$ is calculated for
$\lambda$ satisfying $\langle\lambda+\rho,\alpha_0^\vee\rangle<(k+1)p$
in terms of $p$; also when $\lambda$ is in the Jantzen region. The
next lemma gives the maximum value of $d(\lambda)$ for $\lambda\in
X_1$. For $p\geq h$ this value is independent of $p$ and is an upper
bound for $p<h$.

  \begin{lemma}\label{restrictedDLambda} When $p\geq h$, the maximum
    value of $d(\lambda)$ for $\lambda$ in $X_1$ is 
\begin{enumerate}
  \item $\frac{(n-1)n(n+1)}{6}$ for type $A_n$,
  \item $\frac{(n-1)n(4n+1)}{6}$ for types $B_n$ and $C_n$,
  \item $\frac{2(n-2)(n-1)n}{3}$ for type $D_n$,
  \item $120$ for type $E_6$,
  \item $336$ for type $E_7$,
  \item $1120$ for type $E_8$,
  \item $86$ for type $F_4$,
    \item $10$ for type $G_2$.
  \end{enumerate}
In particular, $d(\lambda)<\lfloor h^3/6\rfloor -1$. When $p<h$, the
above numbers are upper bounds for $\max_{\lambda\in X_1}d(\lambda)$. 
   \end{lemma}
   \begin{proof} It does no harm to assume that $p\geq h$, for this
     calculation. Then \cite[II.6.6(1)]{Jan03} gives the length of the
     Weyl group element taking the lowest alcove $C$ to $D$
     as \[d(D)=\sum_{\alpha\in
       \Phi^+}\left\lfloor\frac{\langle\lambda+\rho,\alpha^\vee\rangle}{p}\right\rfloor\]
for $\lambda\in D$.  Since $(p-2)\rho$ is in the interior of the top
restricted alcove, one may now check the formulae of the lemma 
using the data in \cite{Bourb82}. For the exceptional types we
used the program {\tt minexp} from computer package {\tt
  Dynkin}\footnote{See
  \url{http://www.math.rutgers.edu/~asbuch/dynkin/}.}.\end{proof}
  
  \begin{remark} For $p\geq h$, $0\in C$ and so the lemma gives the
    maximum length of an affine Weyl group element $w$ so that $w\cdot
    0$ is restricted. When $p<h$ the stabiliser of $0$ in the affine
    Weyl group will be non-trivial. Thus the maximum length of an
    affine Weyl group element $w\cdot 0$ which is restricted will be
    strictly less than this, meaning that the expressions above
    represent strict upper bounds on this length.\end{remark}


\section{Bounding the length of Weyl modules}

Recall from \cite[II.5.8(a)]{Jan03} that the $\ch L(\nu)$ with $\nu\in
X^+$ are a basis of $\mathbb Z[X(T)]^W$. It follows that if $M$ is a
$G$-module and one writes $\ch M=\sum_{\nu\in X^+}a_\nu\ch L(\nu)$
one has $[M:L(\nu)]=a_\nu$, the multiplicity of $L(\nu)$
as a composition factor of $M$. We write $\ell(M)$ for $\sum\nu a_\nu$

\begin{theorem}\label{weyllength}Let $\lambda\in X^+$ with
  $\langle\lambda+\rho,\alpha_0^\vee\rangle\leq b$. Then the length
  $\ell(V)$ (i.~e.~number of composition factors) of the Weyl module
  $V=V(\lambda)$ is
  bounded by a constant \[\frac{z_p^{d(\lambda)+1}-1}{z_p-1},\] where
  $z_p=d(\lambda)\log_p(b-1).$
In particular $\ell(V)\leq \frac{z_2^{l(w)+1}-1}{z_2-1}$.
\end{theorem}
\begin{proof} We prove this by induction on the value of
  $d(\lambda)$. If $d(\lambda)=0$ then the strong linkage principle
  implies that $V(\lambda)=H^0(\lambda)=L(\lambda)$, thus $\ell(V)=1$
  and  we are done. 

For $\mu\in X(T)$ (not necessarily dominant) let $X(\mu)\in \mathbb N$
be the number of simple characters counted with multiplicity appearing in
a decomposition of $\chi(\mu)$ into simple characters. Since
$\pm\chi(\mu)$ is the character of $V(w\cdot\mu)$ with $w\cdot\mu\in
X^+$ for some $w\in W$, we have $X(\mu)=\ell(V(w\cdot\mu))$. Consider
equation (\ref{sum}) applied to $V(\lambda)$. Clearly, we
have \begin{align*}\ell(V(\lambda))&\leq 1+\sum_{\alpha\in
    \Phi^+}\sum_{0<mp<\langle\lambda+\rho,\alpha^\vee\rangle}v_p(mp)
X(s_{\alpha,mp}\cdot\lambda)\\
&\leq 1+d(\lambda)\max_{\alpha\in\Phi^+}\max_{0<mp<\langle\lambda+\rho,\alpha^\vee\rangle} v_p(mp)
X(s_{\alpha,mp}\cdot\lambda)\tag{by Lemma \ref{estimates}(i)}\\
&\leq 1+d(\lambda)\max_{\alpha\in\Phi^+}\max_{0<mp<\langle\lambda+\rho,\alpha^\vee\rangle} \left\lfloor\log_p(b-1)\right\rfloor
X(s_{\alpha,mp}\cdot\lambda)\tag{by Lemma \ref{estimates}(ii)}\\
&\leq 1+d(\lambda) \left\lfloor\log_p(b-1)\right\rfloor\max_{d(\mu)<d(\lambda)}X(\mu)\tag{by Lemma \ref{estimates}(iii)}\\
&\leq 1+z_p\frac{z_p^{d(\lambda)}-1}{z_p-1}\tag{by inductive hypothesis}\\
&\leq \frac{z_p^{d(\lambda)+1}-1}{z_p-1}
\end{align*}
as required. The remaining statement is clear, since $z_p\leq z_2$.
\end{proof}

\begin{remark} Using \cite[Theorem 1.1]{Boe01} one can find the
  maximum value $d=\max d(\lambda)$ where the maximum is over all
  $\lambda$ satisfying
  $\langle\lambda,\alpha^\vee\rangle<(k+1)p$. Therefore, for such a
  $\lambda$, one may replace $d(\lambda)$ with $d$ in the conclusion
  of Theorem \ref{weyllength}.\end{remark}

Using the estimates for $d(\lambda)$ in Lemma \ref{restrictedDLambda}
we can now get bounds for the lengths of Weyl modules with restricted
high weights in each type. We state as a corollary a coarse version of
this bound which is valid for each type of root system.

\begin{corollary} Suppose $\lambda\in X_1$. Then the length $\ell(V)$
  of the Weyl module $V=V(\lambda)$ is bounded by a
  constant \[\frac{z_p^{h^3/6+1}-1}{z_p-1},\] where $z_p=\lfloor h^3/6
  (1+\log_p(h-1))\rfloor \leq\lfloor h^3/6 (1+\log_2(h-1))\rfloor$; if
  $p\geq h$ then we may take $z_p=h^3/6$.\end{corollary}
\begin{proof} Lemma \ref{restrictedDLambda} tells us that
  $d(\lambda)\leq h^3/6$. Note that the maximum value of
  $b=b(\lambda)$ for $\lambda\in X_1$ occurs when $\lambda=(p-1)\rho$
  is the first Steinberg weight. Then $\langle
  \lambda+\rho,\alpha_0^\vee\rangle=p(h-1)$; so
  that 
\[\lfloor\log_p(b-1)\rfloor\leq\lfloor\log_p(p(h-1))\rfloor=1+\lfloor\log_p(h-1)\rfloor\leq1+\lfloor\log_2(h-1)\rfloor. \qedhere\]
\end{proof}

\section{Bounding the first cohomology groups for algebraic groups}

The point of this section is to reduce the problem of bounding the
dimensions of the spaces $H^1(G,L(\lambda))$ to a question about the
composition factors of Weyl modules with restricted high weights.
\begin{prop}\label{theprop}Let $\lambda=\lambda_0+p\lambda'$ with
  $\lambda_0\in X_1$ and $\lambda'\in X^+$. Then the following
  inequality holds:\[\dim H^1(G,L(\lambda))\leq
  \dim\Hom_G((L(\lambda')^*)^{[1]},H^0(\lambda_0)/L(\lambda_0))+1.\] \end{prop}
\begin{proof} We may as well assume $H^1(G,L(\lambda))\neq 0$. The
  five term exact sequence arising from the Lyndon--Hochschild--Serre
  spectral sequence applied to $G_1\triangleleft G$ implies
  that \begin{equation}\label{5term}\dim H^1(G,L(\lambda))\leq
    \dim\Hom_G(L(\lambda')^*,H^1(G_1,L(\lambda_0))^{[-1]})+\dim
    H^1(G,H^0(G_1,L(\lambda_0))^{[-1]}\otimes
    L(\lambda')).\end{equation} (In fact we have equality by
  \cite[Corollary 2]{Don83}.) 

Suppose first that $\lambda_0=0$. If $p\neq 2$ or $G$ is not of type
$C_n$ then $H^1(G_1,k)=0$, by \cite[II.12.2]{Jan03} so that the first
term on the right-hand side of (\ref{5term}) is zero and $\dim
H^1(G,L(\lambda))=\dim H^1(G,L(\lambda'))=\dim
H^1(G,L(\lambda)^{[-1]})$. Thus, by induction on the length of the
$p$-adic expansion of $\lambda$, we may assume that $\lambda_0\neq
0$. If $p=2$ and $G$ is of type $C_n$ then $H^1(G_1,k)^{[-1]}\cong
L(\omega_1)$, the natural module for $\mathrm{Sp}_{2n}$, by
\cite[loc.~cit.]{Jan03}. Suppose the first term on the right-hand side
of (\ref{5term}) is non-zero. Then $L(\lambda')^*\cong
L(\lambda')\cong L(\omega_1)$. Then the second term on the right-hand
side of (\ref{5term}) is $\dim H^1(G,L(\omega_1))$, which vanishes
(for instance) by the linkage principle, since $\omega_1$ is miniscule
and not linked to $0$. Thus $\dim H^1(G,L(\lambda))=\dim
H^1(G,L(\omega_1)^{[1]})=1$ and we are done. Otherwise the first term
on the right-hand side of (\ref{5term}) is $0$ and we have $\dim
H^1(G,L(\lambda))=\dim H^1(G,L(\lambda'))=\dim
H^1(G,L(\lambda)^{[-1]})$ as before.

Thus we may assume $\lambda_0\neq 0$. Then this time the second term
on the right-hand side of (\ref{5term}) vanishes and we
have \begin{equation}\label{red}\dim H^1(G,L(\lambda))=\dim
  \Hom_G(L(\lambda')^*,H^1(G_1,L(\lambda_0))^{[-1]}).\end{equation}

Consider the right-hand side of (\ref{red}). We have a short exact
sequence of $G$-modules \[H^0(G_1,M)^{[-1]}\to
H^1(G_1,L(\lambda_0))^{[-1]}\to H^1(G_1,H^0(\lambda_0))^{[-1]},\]
where $M=H^0(\lambda_0)/L(\lambda_0)$. Applying
$\Hom_G(L(\lambda')^*,?)$ to this sequence yields a long exact
sequence
containing \begin{align*}\label{longexact}\tag{*}\Hom_G(L(\lambda')^*,H^0(G_1,M)^{[-1]})\to
  \Hom_G&(L(\lambda')^*,H^1(G_1,L(\lambda_0))^{[-1]})\\&\to
  \Hom_G(L(\lambda')^*,H^1(G_1,H^0(\lambda_0))^{[-1]})\end{align*} as
a subsequence. 

Now the first term is isomorphic to
$\Hom_G(k,\Hom_{G_1}((L(\lambda')^*)^{[1]},M))
\cong \Hom_G((L(\lambda')^*)^{[1]},M)$. Furthermore, looking at the
various cases from \cite[\S3 Theorems (A), (B), (C)]{BNP04-Frob} one
sees that $H^1(G_1,H^0(\lambda))^{[-1]}$ is a direct sum of distinct
$H^0(\omega_i)$, where $\omega_i$ is a fundamental dominant
weight.\footnote{The largest number of such summands is three and
  occurs only in the case where $p=2$, $G=D_4$ and $\lambda=\omega_2$;
  under these assumptions $H^1(G_1,H^0(\omega_2))^{[-1]}\cong
  H^0(\omega_1)\oplus H^0(\omega_3)\oplus H^0(\omega_4)$.} Thus the
second term is always at most $1$. Now, since the dimension of the
middle term in (*) is bounded by the sum of the dimensions of the
outer terms, we are done by (\ref{red}).
\end{proof}

\begin{corollary}\label{corbound}For $\lambda\in X^+$, the dimension
  of $H^1(G,L(\lambda))$ is bounded above by the length $\ell(V)$ of a
  Weyl module $V(\lambda)$ with $\lambda\in X_1$.

Hence \[\dim H^1(G,L(\lambda))\leq \frac{z_p^{h^3/6+1}-1}{z_p-1},\]
where $z_p=\lfloor h^3/6 (1+\log_p(h-1))\rfloor \leq\lfloor h^3/6
(1+\log_2(h-1))\rfloor$; if $p\geq h$ then we may take
$z_p=h^3/6$.  \end{corollary}
\begin{proof}This is immediate from Proposition \ref{theprop}: one
  has \[\Hom_G((L(\lambda')^*)^{[1]},H^0(\lambda_0)/L(\lambda_0))+1\leq
  [H^0(\lambda_0)/L(\lambda_0):(L(\lambda')^*)^{[1]}]+1\leq
  \ell(V(\lambda_0)).\]\end{proof}

\section{From $G$-cohomology to $G_\sigma$-cohomology}

Let $\sigma:G\to G$ be a strictly surjective endomorphism of the simple algebraic group $G$. Then
$\sigma=\tau\circ F^r$ for some $r$, where $F$ is a standard Frobenius
map and $\tau$ is a graph automorphism. The fixed points
$G_\sigma=\{g\in G(k):\sigma(g)=g$ is a finite group of Lie type. By a
result of Steinberg, the simple $kG_\sigma$-modules can all be
identified with  the restrictions of (a subset of) $p^r$-restricted
simple $G$-modules. As in \cite[Theorem 7.10]{CPS09} we use a result
of Bendel--Nakano--Pillen to translate our cohomology bound from the
algebraic setting to the finite setting. The result shows that for
each $G$, in all but a finite number of cases of $G_\sigma$, any
cohomology group $H^1(G_\sigma,L)$ with $L$ a simple
$kG_\sigma$-module is isomorphic to $H^1(G,L')$ for some simple
$G$-module $L'$. This allows us to give a bound for
$G_\sigma$-cohomology (in defining characteristic).

We need the following lemma.

\begin{lemma}\label{steinbergbig}
Let $G$ be a simply connected simple algebraic group over an algebraically closed field $k$ of characteristic $p$. Let $\lambda\in X_r$. Then the dimensions of the Weyl module $V(\lambda)$ and its simple head $L(\lambda)$ are
bounded by $p^{r|\Phi^+|}$.

Thus the simple $kG_\sigma$-module of
  largest dimension is the $r$th Steinberg module of dimension
  $p^{r|\Phi^+|}$.
\end{lemma}
\begin{proof}
For the first part it clearly suffices to bound the dimension of $V(\lambda)$, which is given by the formula
\cite[Corollary 24.6]{FH91}
$$
\frac{\prod_{\alpha\in
  \Phi^+}\langle\lambda+\rho,\alpha^\vee\rangle}
{\prod_{\alpha\in \Phi^+}\langle\rho,\alpha^\vee\rangle}.
$$
For $\lambda\in X_r$ and $\alpha$ a simple root
$$
\langle\lambda+\rho,\alpha^\vee\rangle \le p^r
= \langle (p^r-1)\rho+\rho,\alpha^\vee\rangle.
$$
In other words, among the Weyl modules with $p^r$-restricted heads,
the $r$th Steinberg
module ($=V((p^r-1)\rho)$) maximises all
the inner products with the simple roots, and hence any non-negative
linear
combination of simple roots, and thus in particular, the positive roots.
Thus the dimension of a Weyl module with $p^r$-restricted head is bounded by the
dimension of the $r$th Steinberg module, which is clearly
$p^{r|\Phi^+|}$ by the above formula.

The second part is clear since any simple $kG_\sigma$-module is
obtained as the restriction to $G_\sigma$ of a $p^r$-restricted simple
$G$-module.
\end{proof}

\begin{prop}\label{theoremFinite} Let $L$ be a simple
  $kG_\sigma$-module and let $h$ be the associated Coxeter
  number of $G_\sigma$. Suppose further than $H^1(G,L')\leq b$ for all simple
  $G$-modules $L'$. Then 
\[\dim H^1(G_\sigma,L)\leq
  \max\left\{b,\frac{1}{2}\left(h^2(3h-3)^3\right)^{\frac{h^2}{2}}\right\}.\]
\end{prop}
\begin{proof} 
Recall that $\sigma=\tau\circ F^r$ for $F$ a standard
  Frobenius automorphism. Let $q=p^r$ and set
  $s=\left\lfloor\frac{r}{2}\right\rfloor$. The case where $G$ is of
  type $A_1$ follows from \cite[Corollary 4.5]{AJL83}, so assume
  otherwise. Now, if either (i): $r\geq 2$ and $p^{s-1}(p-1)>h$; or
  (ii): $p\geq 3h-3$, we have $\dim H^1(G_\sigma,L)=\dim H^1(G,L')$
  for some simple $G$-module $L'$ by \cite[Theorem 5.5]{BNP06} and
  \cite[Theorem 5.1]{BNPtw} respectively. Thus we are done in either
  of these cases.

Thus we may assume that $p<3h-3$ and either (a) $r=1$ or (b)
$p^{s-1}(p-1)\leq h$. In case (a), this implies that $q<3h-3$. In case
(b), $s\leq\log_p h$ so that $r\leq 2\log_p h+3$. Then $q=p^r\leq
h^2p^3\leq h^2(3h-3)^3$. 

Now by \cite{GH98}, we have $\dim H^1(G_\sigma,L)\leq 1/2\cdot\dim L$
and by Lemma \ref{steinbergbig}, $\dim L$ is at most $q^{|\Phi^+|}$. It
is easy to check the tables in \cite{Bourb82} to see that $|\Phi^+|$
is no bigger than $h^2/2$ (often with agreement). Thus in case (a) we
have $\dim H^1(G_\sigma,L)\leq 1/2\ (3h-3)^{h^2/2}$ and in case (b) we
have \[\dim H^1(G_\sigma,L)\leq 1/2
\left(h^2(3h-3)^3\right)^{\frac{h^2}{2}}.\] In either case, the
theorem holds.
\end{proof}

We may now tackle the proof of our first two main theorems.

\begin{proof}[Proof of Theorem A]
If $G$ is a Ree or Suzuki group, then $H^1(G,V)\leq 2$ by
\cite{Sin92}, \cite{Sin93} or \cite{SinF4}. Otherwise $G$ is the
quotient by $Z(G)$ of some $H_\sigma$ with $H$ a simply connected algebraic group. Hence
we can lift $V$ to a simple module for $H_\sigma$. Then a
Lyndon--Hochschild--Serre spectral sequence argument gives that $\dim
H^1(G_\sigma,V)=\dim H^1(G,V)$. Now the result follows from Theorem
\ref{theoremFinite} in combination with Corollary \ref{corbound}.
\end{proof}

\begin{proof}[Proof of Theorem B]
The case in cross-characteristic using \cite{GT11} is easy: by that result one has $H^1(G,V)\leq |W|+e$ for $e$ the twisted rank of $G$. Now $|W|$ is no bigger than $2^{l}l!$, so $O(\log(|W|+e))=O(\log |W|)=O(l\log l)$. For the defining characteristic case, note that the Coxeter number of $h$ is linear with the rank of
$G$. The theorem then follows by taking $\log$s of both sides of the
inequality in Theorem A.
\end{proof}

\section{If the Lusztig Character Formula holds}
In \cite{CPS09}, the authors prove a bound on $H^1(G,V)$ for $G$ a semisimple
algebraic group. This is generalised in \cite[Lemma 5.2]{PS11} (which
is also easier to read). They
show \[\dim\Ext^1_G(L(\lambda),L(\mu))\leq p^{|\Phi|}\p(2(p-1)\rho).\]
Hence the same bound will work in the case of cohomology (the case
$\lambda=0$). Suppose $p\geq p_0$ such that the Lusztig character
formula holds (such a $p_0$ is guaranteed to exist by
\cite{AJS94}). Then one may replace $p$ in the above expression by the
fixed value $p_0$. Thus they can assert the existence of an (implicit)
bound on the dimensions of $1$-cohomology. If the Lusztig conjecture
holds, one may replace $p$ in the above formula with any prime at
least as big as $h$ (so certainly $2h$ will work).

Approaching the problem from the view of composition factors of Weyl
modules, we improve this bound in the case of cohomology;
in particular, we can remove the dependency on the Kostant
function. Thus, if the Lusztig conjecture were to hold, we achieve a
better bound than that given in Corollary \ref{corbound}.

\begin{prop}\label{lcfcor}
Let $V=V(\lambda)$ be a Weyl module with $p$-restricted head. Then $\ell(V)\leq p^{|\Phi^+|}$.

If the Lusztig character formula holds for all $p\geq p_0$, then
$\ell(V)\leq p_0^{|\Phi^+|}$. In particular, if the Lusztig conjecture
holds, then $\ell(V)\leq (2h)^{|\Phi^+|}\leq (2h)^{h^2/2}\leq
h^{h^2}$.
\end{prop}
\begin{proof} The first part follows from Lemma \ref{steinbergbig}: certainly the
  dimension of $V$ must bound its length. 

For the remainder of the corollary,  Lusztig's character formula
implies a decomposition of the character of a Weyl module in the
principal block (i.e.~those with high weights of the form $w.0$ for
$w\in W_p$) into characters of simple modules corresponding to a fixed finite collection of elements of the affine Weyl group; moreover, this decomposition is
independent of $p$. Using translation functors, one gets a bound for
all $\ell(V)$ for $V$ of the stated type which is independent of $p\geq p_0$. The first part of
the corollary implies that $\ell(V)\leq p_0^{|\Phi^+|}$ works when
$p=p_0$, thus it works in general. Lusztig's conjecture states that
his character formula should hold for all $p\geq h$. Let $p_0$ be the
lowest prime $\geq h$. Then $p_0\leq 2h$ and so $\ell(V)\leq
p_0^{|\Phi^+|}\leq (2h)^{|\Phi^+|}$ as required. The remaining
inequalities are clear.\end{proof}

\begin{proof}[Proof of Theorem C]

Let $H$ be a simply connected, simple algebraic group with the same
root system as $G$. Then $V$ is obtained as the restriction of a
simple module $V$ for $H$. Since the Lusztig conjecture is assumed, we
have, by Corollary \ref{lcfcor}, that $\ell(W)\leq (2h)^{h^2/2}$ for any Weyl module $W$ with a restricted head. Now,
Corollary \ref{corbound} implies that $\dim H^1(H,V)\leq 
(2h)^{h^2/2}$. 

Proposition \ref{theoremFinite} implies that $\dim H^1(G_\sigma,
V)\leq \max
\left\{(2h)^{h^2/2},\left(h^2(3h-3)^3\right)^{\frac{h^2}{2}}\right\}$
for any strictly surjective endomorphism $\sigma$ of $G$. The
arguments of the proof of Theorem A go through as before, and the
first part of Theorem C follows.

For the second part, as in Theorem B, the cross-characteristic case follows from \cite{GT11}. The defining characteristic case is immediate from the above by noting that $h$ is linear with the
rank of $G$ and taking $\log$s of both sides of the inequality in the
Theorem C.
\end{proof}

\subsection*{Acknowledgements} This paper comprises work done at the
AIM conference `Bounding cohomology and growth rates'. We would like to thank the organisers and staff at the conference for their hospitality.
\bibliographystyle{amsalpha}
\bibliography{bib}

\end{document}